\documentclass{birkmult-1}
\usepackage{amsmath,amssymb,graphicx,mathrsfs,hyperref,amsthm,color}

\title
[Wave equation with a discontinuous coefficient depending on time]
{The wave equation with a discontinuous coefficient depending on time only: generalized solutions and propagation of singularities\thanks{thankssss}}
\author
[H. Deguchi, G. H\"{o}rmann and M. Oberguggenberger]
{\small Hideo Deguchi\\
   \texttt{hdegu@sci.u-toyama.ac.jp}\\
   Department of Mathematics, University of Toyama\\
   Gofuku 3190, 930-8555 Toyama, Japan
  \bigskip
  \\
  G\"{u}nther H\"{o}rmann\\
  \texttt{guenther.hoermann@univie.ac.at}\\
Fakult\"{a}t f\"{u}r Mathematik, Universit\"{a}t Wien \\
Nordbergstra\ss e 15, 1090 Wien, Austria
  \bigskip
  \\
  Michael Oberguggenberger\\
  \texttt{michael.oberguggenberger@uibk.ac.at}\\
Arbeitsbereich f\"{u}r Technische Mathematik,
Universit\"{a}t Innsbruck \\
Technikerstra\ss e 13, A-6020 Innsbruck, Austria
}

\thanks{H. Deguchi was  supported by the Austrian Science Fund (FWF), Lise Meitner project M1155-N13.}

\usepackage{amssymb,amsmath,amsfonts,amsthm,mathrsfs}

\renewcommand{\theequation}{\thesection.\arabic{equation}}
\numberwithin{equation}{section}
\newtheorem{theorem}{Theorem}[section]

\theoremstyle{definition}
\newtheorem{definition}[theorem]{Definition}
\newtheorem{remark}[theorem]{Remark}
\newtheorem{example}[theorem]{Example}

\newcommand{\Cinf}{\ensuremath{\mathcal{C}^\infty}}

\newcommand{\D}{\ensuremath{{\mathcal D}}}

\renewcommand{\L}{\mathcal{L}}

\newcommand{\mb}[1]{\ensuremath{\mathbb{#1}}}
\newcommand{\N}{\mb{N}}

\newcommand{\R}{\mb{R}}

\newcommand{\G}{\ensuremath{{\mathcal G}}}

\newcommand{\EM}{\ensuremath{{\mathcal E}_{M}}}

\newcommand{\Neg}{\mathcal{N}}

\newcommand{\Ginf}{\ensuremath{\G^\infty}}

\newcommand{\singsupp}{\operatorname{sing\,supp\hspace{0.5pt}}}

\newfont{\bigmath}{cmr12 at 13pt}

\newfont{\grecomath}{cmmi12 at 15pt}

\newcommand{\p}{\ensuremath{\partial}}
\newcommand{\beq}{\begin{equation}}
\newcommand{\eeq}{\end{equation}}

\newcommand{\eps}{\varepsilon}

\begin{document}

\begin{abstract}
This paper is devoted to the investigation of propagation of singularities in hyperbolic equations with non-smooth coefficients,
using the Colombeau theory of generalized functions. As a model problem, we study the Cauchy problem for the one-dimensional wave equation with a discontinuous coefficient depending on time.
After demonstrating the existence and uniqueness of generalized solutions in the sense of Colombeau to the problem,
we investigate the phenomenon of propagation of singularities, arising from delta function initial data, for the case of a piecewise constant coefficient. We also provide an analysis of the interplay between singularity strength and propagation effects.
Finally, we show that in case the initial data are distributions, the Colombeau solution to the model problem is associated with the piecewise distributional solution of the corresponding transmission problem.
\end{abstract}
%46F30  Generalized functions for nonlinear analysis (Rosinger, Colombeau, nonstandard, etc.)
%35L05  Wave equation
%35A21  Propagation of singularities

\maketitle

\section{Introduction}\label{intro}
This paper is devoted to propagation of singularities in linear hyperbolic partial differential equations with non-smooth coefficients in the framework of the Colombeau theory of generalized functions \cite{C:1984,C:1985}. Such equations involve a nonlinear interaction of coefficient singularities with those in the solution which emanate from the initial data. In the past years, this problem has been investigated intensively in a series of papers, in which the coefficients and the solution have been viewed as elements of the Colombeau algebra of generalized function.

So far, microlocal elliptic regularity is well understood for general equations and systems in this framework \cite{G:2006, GGO:2003, HOP:06}. Further, generalized Fourier integral operators have been introduced in order to describe the wave front set of generalized solutions \cite{G:ISAAC07, GH:2005,  GH:2006, GHO:2009}. These methods have allowed to describe the propagation of singularities in scalar first order hyperbolic equations with generalized function coefficients \cite{D:2011, GO:2011a, HH:2001}. As a new phenomenon, the Colombeau wave front set has been shown to possess a more refined structure than the corresponding distributional wave front set \cite{O:2008}. For a survey of these methods and relations to various classical approaches we refer to \cite{HallerH:2008}.

In the special situation of hyperbolic equations and systems with piecewise constant coefficients (with possibly jumps across smooth hypersurfaces), the problem may be interpreted as a transmission problem and can be solved classically by a piecewise distributional solution. In case of systems or higher order equations, an incoming singularity will split at the interface into a refracted and a reflected wave. While in many cases it has been shown that the Colombeau solution is associated with the distributional solution of the transmission problem \cite{GO:2011b, O:1989}, the task of identifying the refracted and reflected singularity by means of the Colombeau wave front set -- directly in the generalized solution without considering the associated distribution -- has remained open.

The purpose of this paper is to demonstrate -- in a simple model problem -- that the splitting of singularities at an interface can indeed be observed in the Colombeau solution. At the same time, the effect depends on the interplay of the scale of regularization of coefficients and data. We consider the Cauchy problem for the one-dimensional wave equation with a discontinuous coefficient depending on time
\begin{equation}\label{eqn : wave equation}
	\begin{gathered}
		\partial_t^2u - c(t)^2\partial_x^2u = 0,  \qquad t > 0, \ x \in \mathbb{R},\\
		u|_{t = 0} = u_{0},\quad \partial_tu|_{t = 0} = u_{1}, \qquad x \in \mathbb{R}.
	\end{gathered}
\end{equation}
We will seek solutions in the Colombeau algebra $\mathscr{G}([0,\infty) \times \mathbb{R})$ of generalized functions, which will be defined in Section $\ref{sec : 2}$ below. The Colombeau algebras are constructed as factor algebras of nets of smooth functions, depending on a regularization parameter. They are differential algebras; the space of distributions can be imbedded as a subspace using regularization by convolution with a Friedrichs mollifier. The initial data as well as the coefficient will be taken from the algebra $\mathscr{G}(\mathbb{R})$, which contains the space $\mathscr{D}^{\prime}(\mathbb{R})$ of distributions.

As a model situation, we consider the case when $c(t) = c_0 + (c_1 - c_0)H(t-1)$, $u_0 = 0$ and $u_1 = \delta$, where $c_0, c_1 > 0$, $c_0 \ne c_1$, $H$ is the Heaviside function and $\delta$ is the delta function.
The singularity of the initial data propagates from the origin in the two characteristic directions until it hits the coefficient discontinuity at time $t = 1$. As the coefficient has a discontinuity in time, the singular support is expected to split at time $t = 1$ into a transmitted and a refracted part. In the classical setting, the piecewise distributional solution exhibits this splitting.
We will demonstrate the occurrence of the splitting effect in the Colombeau setting as well, using the generalized singular support.
However, in the Colombeau setting the effect depends on the scale in terms of the regularization parameter. If the initial data and the coefficient are regularized on the same scale, both the transmitted and the refracted ray will be shown to belong to the Colombeau singular support. If the coefficient is regularized by means of a slow scale mollifier, only the transmitted ray belongs to the Colombeau singular support.

In order to be able to observe this effect, we have to prove a new existence result for the Cauchy problem (\ref{eqn : wave equation}) which does not restrict the choice of scale for the regularization of the coefficient. The equation can be rewritten as a one-dimensional first-order hyperbolic system with discontinuous coefficients in the Colombeau algebra of generalized functions.
This transformation will be used in the existence and regularity results. We also demonstrate that the Colombeau solution has the piecewise solution of the transmission problem as its associated distribution. That is, the nets defining the Colombeau solution converge to the piecewise distributional solution. This latter argument relies on energy estimates.

The paper is organized as follows: we recall the definition and basic properties of the Colombeau algebra $\mathscr{G}$ in Section $\ref{sec : 2}$.
In Section $\ref{sec : 3}$, we show that problem $(\ref{eqn : wave equation})$ is uniquely solvable in the Colombeau algebra $\mathscr{G}([0,\infty)\times \mathbb{R})$ without restriction on the scale of regularization of the coefficient (Theorem \ref{thm : existence and uniqueness}).
The problem of propagation of singularities is addressed in Section $\ref{sec : 4}$ (Theorem \ref{thm : propagation1}).
The question whether the regularity of the coefficient affects that of the solution is discussed in Section \ref{sec : 5} (Theorem \ref{thm : propagation2}).
In Section \ref{sec : 6}, we show that the Colombeau generalized solutions to problem (\ref{eqn : wave equation}) with arbitrary distributions as initial data admit the piecewise distributional solutions of the corresponding transmission problem as associated distributions (Theorem \ref{thm : associated distribution}).

\section{The Colombeau theory of generalized functions}\label{sec : 2}

We will employ the {\it special Colombeau algebra} denoted by $\G^{s}$ in \cite{GKOS:2001}, which was called the {\it simplified Colombeau algebra} in \cite{B:1990}.
However, here we will simply use the letter $\G$ instead.
Let us briefly recall the definition and basic properties of the algebra $\G$ of generalized functions.
For more details, see \cite{GKOS:2001}.

Let $\Omega$ be a non-empty open subset of $\mathbb{R}^d$.
Let $\Cinf(\Omega)^{(0,1]}$
be the differential algebra of all maps from the interval $(0,1]$ into $\Cinf(\Omega)$.
Thus each element of
$\Cinf(\Omega)^{(0,1]}$
is a family $(u^{\varepsilon})_{\varepsilon \in (0,1]}$ of real valued smooth functions on $\Omega$.
The subalgebra $\EM(\Omega)$ is defined by all elements $(u^{\varepsilon})_{\varepsilon \in (0,1]}$ of
$\Cinf(\Omega)^{(0,1]}$
with the property that, for all $K \Subset \Omega$ and $\alpha \in \mathbb{N}_0^d$, there exists $p \ge 0$ such that
\begin{equation}\label{eqn : moderate}
	\sup_{x \in K} |\partial_{x}^{\alpha} u^{\varepsilon}(x)| = O(\varepsilon^{-p}) \quad {\rm as}\ \varepsilon \downarrow 0.
\end{equation}
The ideal $\Neg(\Omega)$ is defined by all elements $(u^{\varepsilon})_{\varepsilon \in (0,1]}$ of
$\Cinf(\Omega)^{(0,1]}$
with the property that, for all $K \Subset \Omega$, $\alpha \in \mathbb{N}_0^d$ and $q \ge 0$,
\[
	\sup_{x \in K} |\partial_{x}^{\alpha} u^{\varepsilon}(x)| = O(\varepsilon^q) \quad {\rm as}\ \varepsilon \downarrow 0.
\]
{\it The algebra} $\G(\Omega)$ {\it of generalized functions} is defined by the quotient space
\[
	\G(\Omega) = \EM(\Omega) / \Neg(\Omega).
\]
The Colombeau algebra on a closed half space $[0,\infty) \times \mathbb{R}$ is constructed in a similar way.

We use capital letters for elements of $\G$ to distinguish generalized functions from distributions and denote by $(u^{\varepsilon})_{\varepsilon \in (0,1]}$ a representative of $U \in \G$.
Then for any $U$, $V \in \G$ and $\alpha \in \mathbb{N}_0^d$, we can define the partial derivative $\partial^{\alpha} U$ to be the class of $(\partial^{\alpha}u^{\varepsilon})_{\varepsilon \in (0,1]}$ and the product $UV$ to be the class of $(u^{\varepsilon}v^{\varepsilon})_{\varepsilon \in (0,1]}$.
Also, for any $U =$\ class of $(u^{\varepsilon}(t,x))_{\varepsilon \in (0,1]} \in \G([0,\infty)\times\mathbb{R})$, we can define its restriction $U|_{t = 0} \in \G(\mathbb{R})$ to the line $\{t = 0\}$ to be the class of $(u^{\varepsilon}(0,x))_{\varepsilon \in (0,1]}$.

\begin{remark}\label{remark : imbedding}
The algebra $\G(\Omega)$ contains the space $\mathcal{E}^{\prime}(\Omega)$ of compactly supported distributions.
In fact, the map
\[
	f \mapsto {\rm class\ of}\ (f \ast \rho_{\varepsilon}\mid_{\Omega})_{\varepsilon \in (0,1]}
\]
defines an imbedding of $\mathcal{E}^{\prime}(\Omega)$ into $\G(\Omega)$, where
\[
	\rho_{\varepsilon}(x) = \dfrac{1}{\varepsilon^d} \rho \left(\dfrac{x}{\varepsilon}\right)
\]
and $\rho$ is a fixed element of $\mathcal{S}(\mathbb{R}^d)$ such that $\int \rho(x)\,dx = 1$ and $\int x^{\alpha}\rho(x)\,dx = 0$ for any $\alpha \in \mathbb{N}_0^d$, $|\alpha| \ge 1$.
This can be extended in a unique way to an imbedding of the space $\D^{\prime}(\Omega)$ of distributions.
Moreover, this imbedding turns $\Cinf(\Omega)$ into a subalgebra of $\G(\Omega)$.
\end{remark}

\begin{definition}
A generalized function $U \in \G(\Omega)$ is said to be {\it associated with a distribution} $w \in \D^{\prime}(\Omega)$ if it has a representative $(u^{\varepsilon})_{\varepsilon \in (0,1]} \in \EM(\Omega)$ such that
\[
	u^{\varepsilon} \to w \quad {\rm in}\ \D^{\prime}(\Omega) \quad {\rm as} \ \varepsilon \downarrow 0.
\]
We write $U \approx w$ and call $w$ the {\it associated distribution} of $U$ provided $U$ is associated with $w$.
\end{definition}

Regularity theory for linear equations has been based on the subalgebra $\Ginf(\Omega)$ of {\it regular generalized functions} in $\G(\Omega)$ introduced in \cite{O:1992}.
It is defined by all elements which have a representative $(u^{\varepsilon})_{\varepsilon \in (0,1]}$ with the property that, for all $K \Subset \Omega$, there exists $p \ge 0$ such that, for all $\alpha \in \mathbb{N}_0^d$,
\[
	\sup_{x \in K} |\partial_{x}^{\alpha} u^{\varepsilon}(x)| = O(\varepsilon^{-p}) \quad {\rm as}\ \varepsilon \downarrow 0.
\]
We observe that all derivatives of $u^{\varepsilon}$ have locally the same order of growth in $\varepsilon > 0$, unlike elements of $\EM(\Omega)$.
This subalgebra $\Ginf(\Omega)$ has the property $\Ginf(\Omega) \cap \D^{\prime}(\Omega) = \Cinf(\Omega)$, see \cite[Theorem 25.2]{O:1992}.
Hence, for the purpose of describing the regularity of generalized functions, $\Ginf(\Omega)$ plays the same role for $\G(\Omega)$ as $\Cinf(\Omega)$ does in the setting of distributions.
The $\Ginf$-singular support (denoted by $\singsupp_{\Ginf}$) of a generalized function is defined as the complement of the largest open set on which the generalized function is regular in the above sense.

We end this section by recalling the notion of {\it slow scale nets}.
A net $(r^{\varepsilon})_{\varepsilon \in (0,1]}$ is called a slow scale net if
\[
	|r^{\varepsilon}|^{p} = O(\varepsilon^{-1})\quad {\rm as}\ \varepsilon \downarrow 0
\]
for every $p \ge 0$.
A {\it positive slow scale net} is a slow scale net $(r^{\varepsilon})_{\varepsilon \in (0,1]}$ such that $r^{\varepsilon} > 0$ for all $\varepsilon \in (0,1]$.
We refer to \cite{HO:2004} for a detailed discussion of slow scale nets.

\begin{example}\label{ex : delta}
Let $\varphi$ be a fixed element of $\Cinf_0(\mathbb{R})$ such that $\varphi$ is symmetric, $\varphi^{\prime} \ge 0$ on $[-1,0]$, supp\,$\varphi \subset [-1,1]$ and $\int_{\mathbb{R}} \varphi(x)\,dx = 1$.
Put $\varphi_{\varepsilon}(x) = \varphi(x/\varepsilon)/\varepsilon$.
Then $U \in \G(\mathbb{R})$ defined by the class of $(\varphi_{\varepsilon})_{\varepsilon \in (0,1]}$ is associated with the delta function $\delta$, and $\singsupp_{\Ginf} U = \{0\}$.
On the other hand, if $U \in \G(\mathbb{R})$ is defined as the class of $(\varphi_{h(\varepsilon)})_{\varepsilon \in (0,1]}$, where $(1/h(\varepsilon))_{\varepsilon \in (0,1]}$ is a positive slow scale net, then it is associated with the delta function again, but $\singsupp_{\Ginf} U = \emptyset$.
More generally, for any distribution $f \in \D^{\prime}(\Omega)$, there exists a generalized function $U \in \Ginf(\Omega)$ which is associated with $f$, see e.g. \cite{CH:1994}.
Thus, any distribution on $\Omega$ can be interpreted as an element of $\Ginf(\Omega)$ in the sense of association.
\end{example}

\section{Existence and uniqueness of generalized solutions}\label{sec : 3}
We rewrite problem $(\ref{eqn : wave equation})$ in the form
\begin{equation}\label{eqn : generalized wave equation}
	\begin{gathered}
		\partial_t^2U - C^2\partial_x^2U = 0 \qquad \mbox{in}\ \G([0,\infty)\times\mathbb{R}),\\
		U|_{t = 0} = U_0,\quad \partial_tU|_{t = 0} = U_1 \qquad \mbox{in}\ \G(\mathbb{R})
	\end{gathered}
\end{equation}
in the space of generalized functions, where $C$ is an element of $\G([0,\infty) \times \mathbb{R})$.
In the Colombeau setting, existence and uniqueness follows, for example, from results in \cite{LO:1991, O:1989}, provided
the coefficient $C$ is of logarithmic type, i.e., satisfies bounds of type $O(\log|\varepsilon|)$ in (\ref{eqn : moderate}).
When the coefficient depends on time only, this hypothesis is not required, as we are going to show in the following existence and uniqueness theorem for problem $(\ref{eqn : generalized wave equation})$.

\begin{theorem}\label{thm : existence and uniqueness}
Assume that $C \in \G([0,\infty)\times\mathbb{R})$ has a representative $(c^{\varepsilon}(t))_{\varepsilon \in (0,1]}$ independent of $x$ and satisfying the following two conditions:
\begin{itemize}
\item[(i)] there exist two constants $c_0$, $c_1 > 0$ such that, for any $\varepsilon \in (0,1]$ and $t \ge 0$,
\[
	c_1 \ge c^{\varepsilon}(t) \ge c_0 > 0;
\]
\item[(ii)] for any $\varepsilon \in (0,1]$,
\[
	\int_{0}^{\infty} |(c^{\varepsilon})^{\prime}(t)|\,dt < \infty.
\]
\end{itemize}
Then for any initial data $U_0$, $U_1 \in \G(\mathbb{R})$, problem $(\ref{eqn : generalized wave equation})$ has a unique solution $U \in \G([0,\infty) \times \mathbb{R})$.
\end{theorem}

\begin{proof}
Put $V= \partial_t U - C\partial_xU$ and $W= \partial_tU + C\partial_xU$.
Then problem $(\ref{eqn : generalized wave equation})$ can be rewritten as the Cauchy problem for a first-order hyperbolic system
\begin{equation}\label{eqn : generalized first-order hyperbolic system}
	\begin{array}{rclcl}
		(\partial_t + C\partial_x)V &=& MV-MW &\ & \mbox{in}\ \G([0,\infty)\times\mathbb{R}),\\[2pt]
		(\partial_t - C\partial_x)W &=& MW-MV && \mbox{in}\ \G([0,\infty)\times\mathbb{R}),\\[2pt]
		V|_{t=0}\ =\ V_0 &=& U_1 - (C|_{t=0})U_0^{\prime},&& \mbox{in}\ \G(\mathbb{R}),\\[2pt]
        W|_{t=0}\ =\ W_0&=&U_1+ (C|_{t=0})U_0^{\prime} && \mbox{in}\ \G(\mathbb{R}),
	\end{array}
\end{equation}
where $M = C^{\prime}/(2C) \in \G([0,\infty) \times \mathbb{R})$.
If problem $(\ref{eqn : generalized first-order hyperbolic system})$ has a unique solution $(V,W)$, then so does problem $(\ref{eqn : generalized wave equation})$.
To prove the existence of a solution $(V,W)$, let $(v^{\varepsilon},w^{\varepsilon})$ be the unique $\Cinf$-solution to the Cauchy problem
\begin{equation}\label{eqn : representative}
	\begin{array}{rclcl}
		(\partial_t + c^{\varepsilon}(t)\partial_x)v^{\varepsilon} &=& \mu^{\varepsilon}(t)v^{\varepsilon}-\mu^{\varepsilon}(t)w^{\varepsilon}, &\ &t > 0,\ x \in \mathbb{R},\\[2pt]
		(\partial_t - c^{\varepsilon}(t)\partial_x)w^{\varepsilon} &=& \mu^{\varepsilon}(t)w^{\varepsilon}-\mu^{\varepsilon}(t)v^{\varepsilon}, &  &t > 0,\ x \in \mathbb{R},\\[2pt]			 
		v^{\varepsilon}|_{t = 0}\ =\ v_0^{\varepsilon}& = & u_{1}^{\varepsilon} - c^{\varepsilon}(0)(u_0^{\varepsilon})^{\prime},&  & x \in \mathbb{R},\\[2pt]
		w^{\varepsilon}|_{t = 0}\ =\ w_0^{\varepsilon} & =& u_{1}^{\varepsilon} + c^{\varepsilon}(0)(u_0^{\varepsilon})^{\prime}, &  & x \in \mathbb{R},
	\end{array}
\end{equation}
where $(u_0^{\varepsilon})_{\varepsilon \in (0,1]}$, $(u_1^{\varepsilon})_{\varepsilon \in (0,1]}$, $(c^{\varepsilon})_{\varepsilon \in (0,1]}$ and $(\mu^{\varepsilon})_{\varepsilon \in (0,1]}$ are representatives of $U_0$, $U_1$, $C$ and $M$, respectively, such that  $c^{\varepsilon}$ is as in the statement and $\mu^{\varepsilon} = (c^{\varepsilon})^{\prime}/(2c^{\varepsilon})$.
For the existence of such $(v^{\varepsilon},w^{\varepsilon})$, see \cite{O:1989}.
If we show that $(v^{\varepsilon})_{\varepsilon \in (0,1]}$ and $(w^{\varepsilon})_{\varepsilon \in (0,1]}$ belong to $\EM([0,\infty) \times \mathbb{R})$,
their equivalence classes in $\G([0,\infty) \times \mathbb{R})$ will form a solution of problem $(\ref{eqn : generalized first-order hyperbolic system})$.
To show that the zeroth derivatives of $v^{\varepsilon}$ and $w^{\varepsilon}$ satisfy estimate $(\ref{eqn : moderate})$, consider the characteristic curves $\gamma^{\varepsilon}_{+}(t,x,\tau)$ and $\gamma^{\varepsilon}_{-}(t,x,\tau)$ passing through $(t,x)$ at time $\tau = t$ which satisfy
\begin{align*}
	& \partial_{\tau}\gamma^{\varepsilon}_{+}(t,x,\tau) = c^{\varepsilon}(\tau), \qquad \gamma^{\varepsilon}_{+}(t,x,t) = x, \\
	& \partial_{\tau}\gamma^{\varepsilon}_{-}(t,x,\tau) = -c^{\varepsilon}(\tau), \qquad \gamma^{\varepsilon}_{-}(t,x,t) = x.
\end{align*}
Along these characteristic curves, $v^{\varepsilon}$ and $w^{\varepsilon}$ are respectively calculated as
\begin{align}
	v^{\varepsilon}(t,x) & = v_0^{\varepsilon}(\gamma^{\varepsilon}_{+}(t,x,0)) + \int_0^t \mu^{\varepsilon}(s)(v^{\varepsilon}-w^{\varepsilon})(s,\gamma^{\varepsilon}_{+}(t,x,s))\,ds, \label{eqn : v^{varepsilon}}\\
	w^{\varepsilon}(t,x) & = w_0^{\varepsilon}(\gamma^{\varepsilon}_{-}(t,x,0)) + \int_0^t \mu^{\varepsilon}(s)(w^{\varepsilon}-v^{\varepsilon})(s,\gamma^{\varepsilon}_{-}(t,x,s))\,ds. \label{eqn : w^{varepsilon}}
\end{align}
For each $T > 0$, we define $K_T$ as the trapezoidal region with corners $(0,-\xi)$, $(T,-\xi+c_1T)$, $(T,\xi-c_1T)$, $(0,\xi)$.
Using $(\ref{eqn : v^{varepsilon}})$ and $(\ref{eqn : w^{varepsilon}})$, we see that
\begin{align*}
	\|v^{\varepsilon}\|_{L^{\infty}(K_T)} & \le \|v_0^{\varepsilon}\|_{L^{\infty}(K_0)} + \int_0^t |\mu^{\varepsilon}(s)|(\|v^{\varepsilon}\|_{L^{\infty}(K_s)}+\|w^{\varepsilon}\|_{L^{\infty}(K_s)})\,ds, \\
	\|w^{\varepsilon}\|_{L^{\infty}(K_T)} & \le \|w_0^{\varepsilon}\|_{L^{\infty}(K_0)} + \int_0^t |\mu^{\varepsilon}(s)|(\|w^{\varepsilon}\|_{L^{\infty}(K_s)}+\|v^{\varepsilon}\|_{L^{\infty}(K_s)})\,ds.
\end{align*}
We add these two inequalities and apply Gronwall's inequality to get
\begin{align*}
	& \|v^{\varepsilon}\|_{L^{\infty}(K_T)} + \|w^{\varepsilon}\|_{L^{\infty}(K_T)} \le (\|v_0^{\varepsilon}\|_{L^{\infty}(K_0)} + \|w_0^{\varepsilon}\|_{L^{\infty}(K_0)}) \exp\left(2\int_0^t |\mu^{\varepsilon}(s)|\,ds\right).
\end{align*}
On the right-hand side, the terms involving $v_0^{\varepsilon}$ and $w_0^{\varepsilon}$ are of order $O(\varepsilon^{-p})$ for some $p \ge 0$.
The exponential term is uniformly bounded in $\varepsilon$ by the condition (ii) on $c^{\varepsilon}$.
Hence, the zeroth derivatives of $v^{\varepsilon}$ and $w^{\varepsilon}$ satisfy estimate $(\ref{eqn : moderate})$ on $K_T$.
We can obtain analogous estimates for all derivatives of $v^{\varepsilon}$ and $w^{\varepsilon}$ by differentiating the equations and using the same argument.
Thus, $(v^{\varepsilon})_{\varepsilon \in (0,1]}$ and $(w^{\varepsilon})_{\varepsilon \in (0,1]}$ belong to $\EM([0,\infty) \times \mathbb{R})$.

For the proof of uniqueness, we only need to obtain the zero-order estimates (by Lemma 1.2.3 in \cite{GKOS:2001}), which follow along the same line as above.
\end{proof}

We remark that condition (ii) in Theorem \ref{thm : existence and uniqueness} can be weakened to the requirement that $\exp\left(\int_0^{\infty}|\mu^\eps(t)|\,dt\right) = O(\varepsilon^{-p})$ as $\varepsilon \downarrow 0$ for some $p \ge 0$.

\section{Propagation of singularities}\label{sec : 4}

In this section we study the phenomenon of propagation of singularities in the generalized solution to problem $(\ref{eqn : wave equation})$ with $c(t) = c_0 + (c_1 - c_0)H(t-1)$, $u_0 \equiv 0$ and $u_1 = \delta$, where $c_0, c_1 > 0$, $c_0 \ne c_1$, $H$ is the Heaviside function and $\delta$ is the delta function.
Let us begin by regularizing the coefficient and initial data.
Let $\varphi$ be as in Example \ref{ex : delta} and put
\[
	c^{\varepsilon}(t) = (c \ast \varphi_{\varepsilon})(t)= c_0 + (c_1-c_0)\int_{\mathbb{R}} H(t-1-\varepsilon s)\varphi(s)\,ds.
\]
Then $c^{\varepsilon} \to c$ in $\D^{\prime}(\mathbb{R})$ as $\varepsilon \downarrow 0$, and the family $(c^{\varepsilon})_{\varepsilon \in (0,1]}$ belongs to $\EM([0,\infty) \times \mathbb{R})$.
We define $C \in \G([0,\infty) \times \mathbb{R})$ as the class of $(c^{\varepsilon})_{\varepsilon \in (0,1]}$.
Then $\singsupp_{\Ginf} C = \{t=1\}$.
We also define $U_0 \equiv 0$ and $U_1 \in \G(\mathbb{R})$ as the class of $(\varphi_{\varepsilon})_{\varepsilon \in (0,1]}$, where $\varphi$ again is a mollifier as in Example \ref{ex : delta}. (Actually, the mollifier need not be the same as the one chosen for the regularization of the coefficient $c$.)
Thus we interpret problem $(\ref{eqn : wave equation})$ with $c(t) = c_0 + (c_1 - c_0)H(t-1)$, $u_0 \equiv 0$ and $u_1 = \delta$ as problem $(\ref{eqn : generalized wave equation})$ with $C$, $U_0$ and $U_1$ defined above.
The existence and uniqueness of a generalized solution is then ensured by Theorem \ref{thm : existence and uniqueness}.
As may be seen from the following theorem, the splitting of the singularities occurs at points of discontinuity of the coefficient (see Figure \ref{fig : singular support1 of U}).

\begin{figure}[htbp]
\begin{center}
\includegraphics{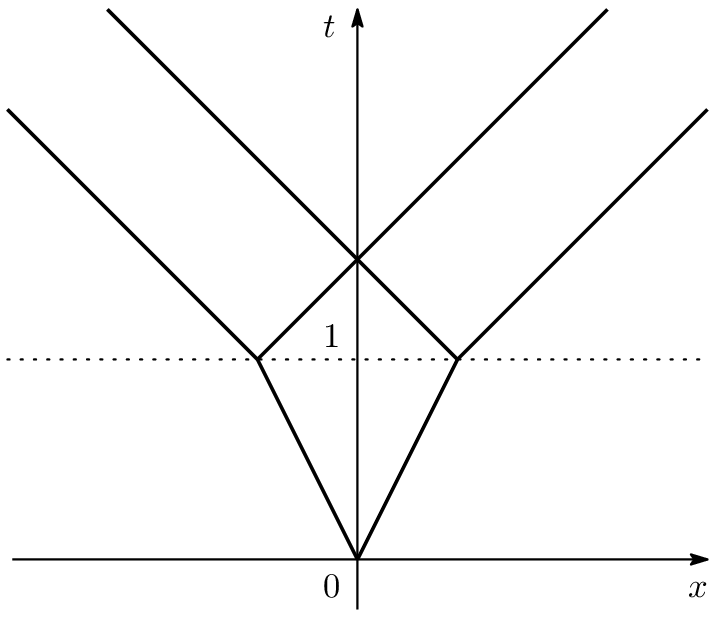}
\end{center}
\caption{The $\Ginf$-singular support of the solution $U$}
\label{fig : singular support1 of U}
\end{figure}

\begin{theorem}\label{thm : propagation1}
Let $C$, $U_0$ and $U_1$ be as above and let $U \in \G([0,\infty) \times \mathbb{R})$ be the solution to problem $(\ref{eqn : generalized wave equation})$.
Then it holds that
\begin{equation}
   \begin{array}{rl}
	\singsupp_{\Ginf} U
	= &\!\!\!\!\left\{(t,x) \mid x = \pm\int_0^t c(s)\,ds,\ t \ge 0\right\} \\[4pt]
	&\!\!\!\!\cup \left\{(t,x) \mid x = \pm\left(2\int_0^1 c(s)\,ds-\int_0^t c(s)\,ds\right),\ t \ge 1\right\}.
    \end{array}
\label{eqn : singular support of U}
\end{equation}
\end{theorem}

\begin{proof}
For the sake of presentation, we first assume that $c_1 > c_0$.
To show that assertion $(\ref{eqn : singular support of U})$ holds, we first calculate the $\Ginf$-singular support of the solution $(V,W)$ to problem $(\ref{eqn : generalized first-order hyperbolic system})$ with $V_0 = U_1$ and $W_0 \equiv 0$.
As may be seen from the proof of Theorem \ref{thm : existence and uniqueness}, the solution $(V,W)$ has a representative $(v^{\varepsilon},w^{\varepsilon})_{\varepsilon \in (0,1]}$, which satisfies the integral equations $(\ref{eqn : v^{varepsilon}})$ with $v_0^{\varepsilon} = \varphi_{\varepsilon}$ and $(\ref{eqn : w^{varepsilon}})$ with $w_0^{\varepsilon} = 0$:
\begin{align*}
	v^{\varepsilon}(t,x) & = \varphi_{\varepsilon}(\gamma_+^{\varepsilon}(t,x,0)) + \int_0^t \mu^{\varepsilon}(s)(v^{\varepsilon}-w^{\varepsilon})(s,\gamma_+^{\varepsilon}(t,x,s))\,ds, \\
	w^{\varepsilon}(t,x) & = \int_0^t \mu^{\varepsilon}(s)(w^{\varepsilon}-v^{\varepsilon})(s,\gamma_-^{\varepsilon}(t,x,s))\,ds,
\end{align*}
where
\begin{align*}
	\gamma_+^{\varepsilon}(t,x,s) = \gamma_+^{\varepsilon}(t,x,0) + \int_0^s c^{\varepsilon}(\tau)\,d\tau,\qquad \gamma_+^{\varepsilon}(t,x,t) = x, \\
	\gamma_-^{\varepsilon}(t,x,s) = \gamma_-^{\varepsilon}(t,x,0) - \int_0^s c^{\varepsilon}(\tau)\,d\tau,\qquad \gamma_-^{\varepsilon}(t,x,t) = x.
\end{align*}
The solution $(v^{\varepsilon},w^{\varepsilon})$ is obtained by iteration, see \cite{O:1992}.
From this, we can check that  $v^{\varepsilon}(t,x) \ge 0$ and $w^{\varepsilon}(t,x) \le 0$ for $t \ge 0$ and $x \in \mathbb{R}$.
Using this fact, we get
\begin{align}
	v^{\varepsilon}(t,x) & \ge \varphi_{\varepsilon}(\gamma_+^{\varepsilon}(t,x,0)) \ge 0, \label{eqn : v}\\
	w^{\varepsilon}(t,x) & \le -\int_0^t \mu^{\varepsilon}(s)v^{\varepsilon}(s,\gamma_-^{\varepsilon}(t,x,s))\,ds \notag \\
	& \le -\int_0^t \mu^{\varepsilon}(s)\varphi_{\varepsilon}(\gamma_+^{\varepsilon}(s,\gamma_-^{\varepsilon}(t,x,s),0))\,ds \notag \\
	& = -\int_0^t \mu^{\varepsilon}(s)\varphi_{\varepsilon}\left(\gamma_-^{\varepsilon}(t,x,0) - 2\int_0^s c^{\varepsilon}(\tau)\,d\tau\right)ds \le 0. \label{eqn : w}
\end{align}
{\bf Step 1}. We first prove that
\begin{equation}\label{eqn : singular support of V}
	\singsupp_{\Ginf} V = \left\{(t,x) \mid x = \int_0^t c(s)\,ds,\ t \ge 0\right\}.
\end{equation}
It is easy to check that $V \in \G([0,\infty) \times \mathbb{R})$ vanishes off $\{(t,x) \mid x = \int_0^t c(s)\,ds,\ t \ge 0\}$ and thus is $\Ginf$-regular there.
We will show first that $\{(t,x) \mid x = \int_0^t c(s)\,ds,\ t \ge 0\}$ is contained in $\singsupp_{\Ginf}V$.

Fix $t > 0$ arbitrarily, and consider
\[
	A_{t,\varepsilon}= \left|\dfrac{v^{\varepsilon}(t,\gamma_+^{\varepsilon}(0,\varepsilon,t)) - v^{\varepsilon}(t,\gamma_+^{\varepsilon}(0,0,t))}{\gamma_+^{\varepsilon}(0,\varepsilon,t) - \gamma_+^{\varepsilon}(0,0,t)}\right|.
\]
By inequality $(\ref{eqn : v})$, we have $v^{\varepsilon}(t,\gamma_+^{\varepsilon}(0,0,t)) \ge \varphi(0)/\varepsilon > 0$.
Noting that $v^{\varepsilon}(t,x) = 0$ for $x \ge \gamma_+^{\varepsilon}(0,\varepsilon,t)$, we get $v^{\varepsilon}(t,\gamma_+^{\varepsilon}(0,\varepsilon,t)) = 0$.
Furthermore, $\gamma_+^{\varepsilon}(0,\varepsilon,t) - \gamma_+^{\varepsilon}(0,0,t) = \varepsilon$.
Hence, we see that $A_{t,\varepsilon} \ge \varphi(0)/\varepsilon^2$.
The mean value theorem shows that there exists $x_1^{\varepsilon} \in (\gamma_+^{\varepsilon}(0,0,t),\gamma_+^{\varepsilon}(0,\varepsilon,t))$ such that
\[
	|\partial_x v^{\varepsilon}(t,x_1^{\varepsilon})| \ge \dfrac{\varphi(0)}{\varepsilon^2}.
\]
Repeat this process to find that for any $n \ge 2$ there exists $x_n^{\varepsilon} \in (x_{n-1}^{\varepsilon},\gamma_+^{\varepsilon}(0,\varepsilon,t))$ such that
\[
	\left|\partial_x^n v^{\varepsilon}(t,x_n^{\varepsilon})\right| \ge \dfrac{\varphi(0)}{\varepsilon^{n+1}}.
\]
Since $x_n^{\varepsilon} \to \int_0^t c(s)\,ds$ as $\varepsilon \downarrow 0$, assertion $(\ref{eqn : singular support of V})$ holds.
\smallskip\\
{\bf Step 2}. We next prove that
\begin{equation}\label{eqn : singular support of W}
	\singsupp_{\Ginf} W = \left\{(t,x) \mid x = 2\int_0^1 c(s)\,ds-\int_0^t c(s)\,ds,\ t \ge 1\right\}.
\end{equation}
The proof is similar.
We can easily check that $W$ equals $0$ outside $\{(t,x) \mid x = 2\int_0^1 c(s)\,ds-\int_0^t c(s)\,ds,\ t \ge 1\}$ and thus is $\Ginf$-regular there.
We will show below that $\{(t,x) \mid x = 2\int_0^1 c(s)\,ds-\int_0^t c(s)\,ds,\ t \ge 1\} \subset$ $\singsupp_{\Ginf}W$.

Fix $t > 1$ arbitrarily.
Let $\varepsilon < t - 1$, and consider
\begin{equation*}
	B_{t,\varepsilon}= \left|\dfrac{w^{\varepsilon}(t,\gamma_-^{\varepsilon}(1, \int_0^{1}c^{\varepsilon}(\tau)\,d\tau,t)) - w^{\varepsilon}(t,\gamma_-^{\varepsilon}(1-\varepsilon, -\varepsilon + \int_0^{1-\varepsilon}c^{\varepsilon}(\tau)\,d\tau,t))}{\gamma_-^{\varepsilon}(1, \int_0^{1}c^{\varepsilon}(\tau)\,d\tau,t) - \gamma_-^{\varepsilon}(1-\varepsilon, -\varepsilon + \int_0^{1-\varepsilon}c^{\varepsilon}(\tau)\,d\tau,t)}\right|.
\end{equation*}
We see that, for $s \ge 0$,
\begin{equation}\label{eqn : curve1}
	\gamma_-^{\varepsilon}\left(1,\int_0^1 c^{\varepsilon}(\tau)\,d\tau,s\right) = 2\int_0^1 c^{\varepsilon}(\tau)\,d\tau - \int_0^s c^{\varepsilon}(\tau)\,d\tau.
\end{equation}
Using this, inequality $(\ref{eqn : w})$ and supp\,$\mu^{\varepsilon} \subset [1-\varepsilon,1+\varepsilon]$, we get
\begin{align}
	w^{\varepsilon}\left(t,\gamma_-^{\varepsilon}\left(1, \int_0^{1}c^{\varepsilon}(\tau)\,d\tau,t\right)\right)
	& \le -\int_{1-\varepsilon}^{1+\varepsilon} \mu^{\varepsilon}(s)\varphi_{\varepsilon}\left(2\int_s^1 c^{\varepsilon}(\tau)\,d\tau\right)ds.\label{eqn : w2}
\end{align}
Choose $a \in (0,2c_1)$ so that $\varphi(a) > 0$.
Then, for $1 - \varepsilon \le 1 - (a\varepsilon)/(2c_1) \le s \le 1 + (a\varepsilon)/(2c_1) \le 1 + \varepsilon$, we have $|2\int_{s}^1 c^{\varepsilon}(\tau)\,d\tau| \le a\varepsilon$.
Using this and noting that $\varphi^{\prime} \ge 0$ on $[-1,0]$ and the symmetry of $\varphi$, we find that $\varphi_{\varepsilon}(2\int_{s}^1 c^{\varepsilon}(\tau)\,d\tau) \ge \varphi_{\varepsilon}(a\varepsilon)$ for $1 - (a\varepsilon)/(2c_1) \le s \le 1 + (a\varepsilon)/(2c_1)$.
Therefore, by inequality $(\ref{eqn : w2})$, we obtain
\begin{align*}
	w^{\varepsilon}\left(t,\gamma_-^{\varepsilon}\left(1, \int_0^{1}c^{\varepsilon}(\tau)\,d\tau,t\right)\right)
	& \le -\int_{1-(a\varepsilon)/(2c_1)}^{1 + (a\varepsilon)/(2c_1)} \mu^{\varepsilon}(s)\varphi_{\varepsilon}(a\varepsilon)\,ds\\
	& = -\dfrac{1}{2}\log \dfrac{c^{\varepsilon}(1+(a\varepsilon)/(2c_1))}{c^{\varepsilon}(1-(a\varepsilon)/(2c_1))} \cdot \dfrac{\varphi(a)}{\varepsilon},
\end{align*}
where
\[
	\beta=\dfrac{c^{\varepsilon}(1+(a\varepsilon)/(2c_1))}{c^{\varepsilon}(1-(a\varepsilon)/(2c_1))}
	= \dfrac{c_0 + (c_1-c_0)\int_{-\infty}^{a/(2c_1)} \varphi(y)\,dy}{c_0 + (c_1-c_0)\int_{-\infty}^{-a/(2c_1)} \varphi(y)\,dy}
\]
is independent of $\varepsilon$ and greater than $1$.
Thus, we obtain
\[
	w^{\varepsilon}\left(t,\gamma_-^{\varepsilon}\left(1, \int_0^{1}c^{\varepsilon}(\tau)\,d\tau,t\right)\right) \le -\dfrac{1}{2}\log \beta \cdot \dfrac{\varphi(a)}{\varepsilon} < 0.
\]
We also see that, for $s \ge 0$,
\begin{equation}\label{eqn : curve2}
	\gamma_-^{\varepsilon}\left(1-\varepsilon, -\varepsilon + \int_0^{1-\varepsilon}c^{\varepsilon}(\tau)\,d\tau,s\right) = -\varepsilon + 2\int_0^{1-\varepsilon}c^{\varepsilon}(\tau)\,d\tau - \int_0^{s}c^{\varepsilon}(\tau)\,d\tau,
\end{equation}
on which $w^{\varepsilon} \equiv 0$.
Furthermore, by $(\ref{eqn : curve1})$ and $(\ref{eqn : curve2})$, we have
\begin{align*}
	& \gamma_-^{\varepsilon}\left(1, \int_0^{1}c^{\varepsilon}(\tau)\,d\tau,t\right) - \gamma_-^{\varepsilon}\left(1-\varepsilon, -\varepsilon + \int_0^{1-\varepsilon}c^{\varepsilon}(\tau)\,d\tau,t\right) \\
	& \hspace{1cm} = \varepsilon + 2\int_{1-\varepsilon}^{1}c^{\varepsilon}(\tau)\,d\tau
	\le (1 + 2c_1)\varepsilon.
\end{align*}
Hence, $B_{t,\varepsilon}  \ge (\log \beta \cdot \varphi(a))/(2(1 + 2c_1)\varepsilon^2)$.
By the mean value theorem, there exists $x_1^{\varepsilon} \in (\gamma_-^{\varepsilon}(1-\varepsilon, -\varepsilon + \int_0^{1-\varepsilon}c^{\varepsilon}(\tau)\,d\tau,t), \gamma_-^{\varepsilon}(1, \int_0^{1}c^{\varepsilon}(\tau)\,d\tau,t))$ such that
\[
	|\partial_x w^{\varepsilon}(t,x_1^{\varepsilon})| \ge \dfrac{\log \beta \cdot \varphi(a)}{2(1 + 2c_1)\varepsilon^2}.
\]
Repeating this process gives
that for any $n \ge 2$, there exists $x_n^{\varepsilon} \in (\gamma_-^{\varepsilon}(1-\varepsilon, -\varepsilon + \int_0^{1-\varepsilon}c^{\varepsilon}(\tau)\,d\tau,t), x_{n-1}^{\varepsilon})$ such that
\[
	|\partial_x^n w^{\varepsilon}(t,x_{n}^{\varepsilon})| \ge \dfrac{\log \beta \cdot \varphi(a)}{2(1 + 2c_1)^n\varepsilon^{n+1}}.
\]
Since $x_n^{\varepsilon} \to 2\int_0^1 c(s)\,ds-\int_0^t c(s)\,ds$ as $\varepsilon \downarrow 0$, assertion $(\ref{eqn : singular support of W})$ holds.
\smallskip\\
{\bf Step 3}. We finally prove that assertion $(\ref{eqn : singular support of U})$ holds.
Similarly to Steps 1 and 2, we can show that
\begin{align}
	\singsupp_{\Ginf} V
	& = \left\{(t,x) \mid x = -\left(2\int_0^1 c(s)\,ds-\int_0^t c(s)\,ds\right),\ t \ge 1\right\}, \label{eqn : singular support2 of V}\\
	\singsupp_{\Ginf} W
	& = \left\{(t,x) \mid x = -\int_0^t c(s)\,ds,\ t \ge 0\right\}, \label{eqn : singular support2 of W}
\end{align}
if $V_0 \equiv 0$ and $W_0 = U_1$.
By the definitions of $V$ and $W$, the $\Ginf$-singular support of $U$ coincides with the union of  $(\ref{eqn : singular support of V})$, $(\ref{eqn : singular support of W})$, $(\ref{eqn : singular support2 of V})$ and $(\ref{eqn : singular support2 of W})$.
Thus, assertion $(\ref{eqn : singular support of U})$ follows.

The case $c_0 > c_1$ can be treated by the same arguments, using a change of sign in $U_1$ and invoking the linearity of the equation. The proof of Theorem \ref{thm : propagation1} is now complete.
\end{proof}

We remark that assertion $(\ref{eqn : singular support of U})$ in Theorem \ref{thm : propagation1} still holds when $U_1 \in \G(\mathbb{R})$ is given by the class of $(\varphi_{\varepsilon}^n)_{\varepsilon \in (0,1]}$ with $n \in \N$, i.e., $U_1$ is any power of the delta function.

\section{Slow scale coefficients and regularity along the refracted ray}\label{sec : 5}

We here discuss how the regularity of the coefficient $C$ affects that of the solution $U$ to problem $(\ref{eqn : generalized wave equation})$. As mentioned in Example \ref{ex : delta},
one may regularize the piecewise constant propagation speed $c(t) = c_0 + (c_1 - c_0)H(t-1)$ in such a way that the corresponding element $C\in\G(\R)$ belongs to $\Ginf(\R)$. In fact, it suffices to use a mollifier $\varphi_{h(\eps)}$, where $(1/h(\eps))_{\varepsilon \in (0,1]}$ is a positive slow scale net. It has been shown in \cite{HOP:06} that a bounded element $C$ of $\G(\R)$ belongs to $\Ginf(\R)$ if and only if all derivatives satisfy slow scale bounds.
Consider problem $(\ref{eqn : generalized wave equation})$ with $U_0 \equiv 0$ and $U_1$
given by the class of $(\varphi_{\varepsilon})_{\varepsilon \in (0,1]}$ as in Theorem $\ref{thm : propagation1}$.
Then as may be seen from the following theorem, the refracted rays
\begin{eqnarray*} % \label{eqn : refracted ray}
	\Gamma_{-} &=& \left\{(t,x) \mid x = 2\int_0^1 c(s)\,ds - \int_0^t c(s)\,ds,\ t > 1\right\}\\
     &=& \left\{(t,x) \mid x = c_0 - c_1(t-1),\ t > 1\right\}\\[2pt]
    \Gamma_{+} &=& \left\{(t,x) \mid x = -2\int_0^1 c(s)\,ds + \int_0^t c(s)\,ds,\ t > 1\right\}\\
     &=& \left\{(t,x) \mid x = -c_0 + c_1(t-1),\ t > 1\right\}
\end{eqnarray*}
do not belong to $\singsupp_{\Ginf} U$, if $C$ belongs to $\Ginf(\R)$ (see Figure \ref{fig : singular support2 of U}).

\begin{figure}[htbp]
\begin{center}
\includegraphics{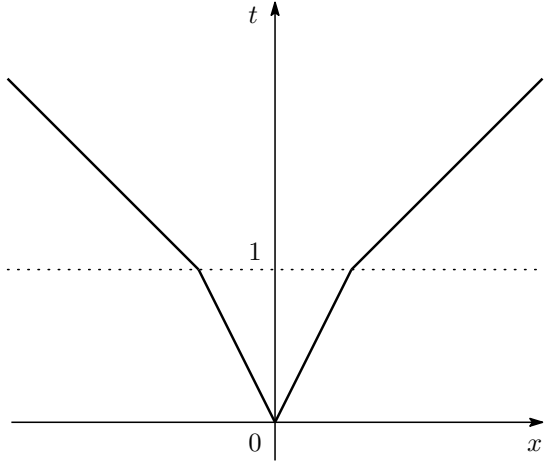}
\end{center}
\caption{The $\Ginf$-singular support of the solution $U$ with slow scale coefficient}
\label{fig : singular support2 of U}
\end{figure}

\begin{theorem}\label{thm : propagation2}
Let $C$, $U_0$ and $U_1$ be as above and let $U \in \G([0,\infty) \times \mathbb{R})$ be the solution to problem
$(\ref{eqn : generalized wave equation})$.
If $C\in\Ginf(\R)$, then
\[
	\singsupp_{\Ginf} U = \left\{(t,x) \mid x = \pm\int_0^t c(s)\,ds,\ t \ge 0\right\}.
\]
\end{theorem}

\begin{proof}
As can be seen from the proof of Theorem \ref{thm : propagation1}, it suffices to consider problem
$(\ref{eqn : generalized first-order hyperbolic system})$ when $V_0 = U_1$ and $W_0 \equiv 0$:
\begin{equation}\label{eq:system}
\begin{gathered}
(\p_t + C\p_x)V = MV - MW \qquad {\rm in\ } \G([0,\infty)\times\R),\\
(\p_t - C\p_x)W = MW - MV \qquad {\rm in\ } \G([0,\infty)\times\R),\\
V|_{t=0} = U_1,\quad W|_{t=0} = 0\qquad {\rm in\ } \G(\R).
\end{gathered}
\end{equation}
The same argument as in the proof of Theorem \ref{thm : propagation1} shows that
\[
	\singsupp_{\Ginf} V = \left\{(t,x) \mid x = \int_0^t c(s)\,ds,\ t \ge 0\right\}
\]
holds.
Thus, it remains to show that $W$ is $\Ginf$-regular in a neighborhood of $\Gamma_{-}$.

A simple calculation shows that the commutator of $\p_t - C(t)\p_x$ and $\p_t + C(t)\p_x$ is given by
\[
   \left[ \p_t - C(t)\p_x, \p_t + C(t)\p_x\right] = 2C'(t)\p_x.
\]
Using this commutator relation and the second line of system (\ref{eq:system}), we have
\begin{align*}
(\p_t - C\p_x)&(\p_t + C\p_x)W = (\p_t + C\p_x)(\p_t - C\p_x)W + 2C'\p_xW\\
    &= (\p_t + C\p_x)\big(M(W-V)\big) + 2C'\p_xW\\
    &= M(\p_t + C\p_x)W - M(\p_t + C\p_x)V + M'(W-V) + 2C'\p_xW.
\end{align*}
From the first line of system (\ref{eq:system}),
\[
   (\p_t + C\p_x)V = M(V-W).
\]
Further,
\[
   \p_xW = \frac1{2C}\Big((\p_t + C\p_x)W - (\p_t - C\p_x)W\Big) = \frac1{2C}\Big((\p_t + C\p_x)W + M(V-W)\Big).
\]
Together with $C'/(2C) = M$ this leads to
\[
   (\p_t - C\p_x)(\p_t + C\p_x)W = 3M(\p_t + C\p_x)W +(M^2-M')(V-W).
\]
Let $K_T$ be as in Theorem \ref{thm : propagation1}.
Then the representatives $(v^\eps,w^\eps)$ of $(V,W)$ are of order $\eps^{-N}$ on $K_T$ for some $N$. Further, $M$ and its derivatives are slow scale by assumption;
hence $|(\mu^\eps)^2 - (\mu^\eps)'| \leq \lambda_1(\eps)$ for some slow scale net $(\lambda_1(\eps))_{\eps \in (0,1]}$.
From integration along characteristics and Gronwall's inequality as in the proof of Theorem \ref{thm : existence and uniqueness}, we get
\[
   \|(\p_t + c^\eps\p_x)w^\eps\|_{L^\infty(K_T)} \leq \lambda_1(\eps)\eps^{-N} T\exp\left(\int_0^T 3|\mu^\eps(s)|\,ds\right) = {\mathcal O}(\lambda_1(\eps)\eps^{-N})
\]
since $\int_0^T |\mu^\eps(s)|\,ds$ is bounded by construction.
Recursively we find that
\[
 (\p_t - C\p_x)(\p_t + C\p_x)^k W = (2k+1)(\p_t + C\p_x)^k W + L(V,W,M,C)
\]
where $L$ depends linearly on $V$, $W$, $(\p_t + C\p_x)W, \ldots, (\p_t + C\p_x)^{k-1}W$ and may contain derivatives of $C$, $M$ and their products and powers.
The linearity of $L$ in the first variables and the slow scale property of $C,M$ and their derivatives yield by induction
that
\[
\|(\p_t + c^\eps\p_x)^{k}w^\eps\|_{L^\infty(K_T)} = {\mathcal O}(\lambda_k(\eps)\eps^{-N})
\]
for some slow scale net $(\lambda_k(\eps))_{\eps \in (0,1]}$.

For the remainder of the proof we may use the fact -- already shown in Theorem \ref{thm : propagation1} -- that $V$ vanishes on $\Gamma_{-}$, hence
\[
   (\p_t - C\p_x)W = MW, \quad (\p_t - C\p_x)^2W = M(\p_t - C\p_x)W + M'W, \ldots
\]
along $\Gamma_{-}$. Again the higher order directional derivatives of $W$ in the direction $\p_t - C\p_x$ can be estimated inductively and bounded by some
slow scale net times $\eps^{-N}$. Similarly, all mixed derivatives $(\p_t \pm C\p_x)^k(\p_t \mp C\p_x)^\ell W$ can be estimated by reduction to previously
computed terms. In conclusion, all derivatives of $W$ can be bounded by some slow scale net times $\eps^{-N}$, hence are of order ${\mathcal O}(\eps^{-N-1})$.
Thus $W$ is $\Ginf$ along $\Gamma_{-}$.
\end{proof}

We remark that the result can be generalized in various ways. For example, the integral bound on $|\mu^\eps|$ can be replaced by the requirement that the net $\exp \left(\int_0^{\infty}|\mu^\eps(t)|\,dt\right)$ is slow scale. The proof is expected to go through for arbitrary initial data $U_1$ whose $\Ginf$-singular support is $\{0\}$.

\section{Associated distributions}\label{sec : 6}

We consider again the wave equation in one space dimension with propagation speed $c = c(t)$ depending on time, i.e.,
\begin{equation}\label{eq:wave1D}
   \p^2_t u - c(t)^2\p^2_x u = 0
\end{equation}
with initial conditions
\begin{equation}\label{eq:ICwave1D}
   u|_{t=0} = u_0,\qquad \p_t u|_{t=0} = u_1.
\end{equation}
If $c(t)$ is piecewise constant, say $c(t) = c_0 + (c_1 - c_0)H(t-1)$, and $u_0, u_1 \in \D'$, this problem has a unique solution
\begin{equation}\label{eq:piecewise}
   u \in {\mathcal C}^1([0,\infty):\D'(\R)) \cap \big({\mathcal C}^2((0,1):\D'(\R)) \oplus {\mathcal C}^2((1,\infty):\D'(\R))\big).
\end{equation}
Thus the problem is interpreted as a transmission problem across a discontinuity at $t=1$; the transmission condition is that the solution should be a continuously differentiable function of time with values in $\D'(\R)$.
It is simply obtained by solving the wave equation for $t < 1$ and for $t > 1$ and taking the terminal values at $t = 1$ as initial values for $t>1$.

On the other hand, imbedding $u_0$, $u_1$ and $c$ into the Colombeau algebra by convolution with suitable compactly supported mollifiers (concerning $c$, we use a mollifier as in Example \ref{ex : delta}), problem (\ref{eq:wave1D}), (\ref{eq:ICwave1D}) has a unique solution $U \in \G([0,\infty)\times\R)$. We are going to show that the Colombeau solution is associated with
the piecewise distributional solution.

\begin{theorem}\label{thm : associated distribution}
Let $c(t)$ be a piecewise constant, strictly positive function and let $u_0, u_1 \in \D'(\R)$. Then the corresponding generalized solution $U \in \G([0,\infty)\times\R)$ is associated with the piecewise distributional solution $u$ satisfying (\ref{eq:piecewise}).
\end{theorem}

\begin{proof}
We first assume that $u_0 \in {\mathcal C}^2(\R)$ and $u_1 \in {\mathcal C}^1(\R)$ and that both have compact support.
Fix some $T > 1$. Let $(u^\eps)_{\eps \in (0,1]}$ be a representative of the generalized solution that vanishes for $x$ outside some compact set, independently of $t$, $0\leq t \leq T$.
This and the smoothness imply that $u^\eps$ belongs to ${\mathcal C}^\infty\big([0,T):H^\infty(\R)\big)$. Thus we can use energy estimates. Multiplying the wave equation
$u^\eps_{tt} - c^\eps(t)^2u^\eps_{xx} = 0$ by $u^\eps_t$ and integrating by parts we get
\[
    \int_\R\big(u^\eps_{tt}u^\eps_t + (c^\eps)^2 u^\eps_xu^\eps_{xt}\big)\,dx = 0.
\]
Observing that
\[
    \frac12 \frac{d}{dt}(c^\eps u^\eps_x)^2 = (c^\eps)^2 u^\eps_xu^\eps_{xt} + c^\eps (c^\eps)'(u^\eps_x)^2
    \]
we obtain that
\[
   \frac12\frac{d}{dt}\int_\R\big(|u^\eps_t|^2 + (c^\eps)^2 |u^\eps_x|^2\big)\,dx = \int_\R c^\eps (c^\eps)'|u^\eps_x|^2\,dx
\]
for $t\in[0,T)$ and thus the energy estimate
\[
   \int_\R\big(|u^\eps_t|^2 + (c^\eps)^2 |u^\eps_x|^2\big)\,dx \leq \int_\R\big(|u^\eps_{1}|^2 + (c^\eps)^2 |u^\eps_{0x}|^2\big)\,dx
   + 2\int_0^t |c^\eps (c^\eps)'|\int_\R |u^\eps_x|^2\,dx\,dt.
\]
By assumption, $c^\eps (c^\eps)'$ is bounded in $L^1(0,T)$, and so Gronwall's inequality shows that $\int_\R|u^\eps_x(t,x)|^2\,dx$ and hence also $\int_\R|u^\eps_t(t,x)|^2\,dx$ remain
bounded on $[0,T]$, uniformly in $\eps$.
In particular, the family $(u^\eps)_{\eps \in (0,1]}$ is bounded in ${\mathcal C}\big([0,T):H^1(\R)\big)$ as well as in ${\mathcal C}^1\big([0,T):L^2(\R)\big)$. By construction, the supports of the functions $u^\varepsilon$ are contained in a common bounded set. Thus the first property implies that $(u^\eps(\cdot,t))_{\eps \in (0,1]}$ is relatively compact in $L^2(\R)$ for every $t$. The second property shows that $(u^\eps)_{\eps \in (0,1]}$ is an equicontinuous subset of
${\mathcal C}\big([0,T):L^2(\R)\big)$. By Ascoli's theorem, the net $(u^\eps)_{\eps \in (0,1]}$ is relatively compact in ${\mathcal C}\big([0,T):L^2(\R)\big)$.

Taking $x$-derivatives in equation (\ref{eq:wave1D}) we can apply the same argument to $u^\eps_x$ and conclude that both $\int_\R|u^\eps_{xt}(t,x)|^2\,dx$ and
$\int_\R|u^\eps_{xx}(t,x)|^2\,dx$ remain bounded on $[0,T]$, uniformly in $\eps$. By the differential equation (\ref{eq:wave1D}), the same is true of
$\int_\R|u^\eps_{tt}(t,x)|^2\,dx$. By the same argument as above, $(u^\eps_t)_{\eps \in (0,1]}$ is relatively compact in ${\mathcal C}\big([0,T):L^2(\R)\big)$ and so $(u^\eps)_{\eps \in (0,1]}$ is relatively compact in ${\mathcal C}^1\big([0,T):L^2(\R)\big)$.

There exists a subsequence $(u^{\eps_k})_k$ such that
\[
   \lim_{k\to\infty} u^{\eps_k} = \overline{u} \in {\mathcal C}^1\big([0,T):L^2(\R)\big) \subset {\mathcal C}^1([0,T):\D'(\R)).
\]
But on every compact subinterval of $(0,1)$ and of $(1,T)$, $c^\eps$ is identically equal to $c_0$ or $c_1$, respectively, when $\eps$ is sufficiently small.
This implies that $\overline{u}$ is a distributional solution of the wave equation (\ref{eq:wave1D}) on both strips. Since
$\overline{u} \in {\mathcal C}\big([0,T):D'(\R)\big)$, so is $\overline{u}_{xx}$. From the equation, we get that
\[
   \overline{u} \in \big({\mathcal C}^2((0,1):\D'(\R)) \oplus {\mathcal C}^2((1,\infty):\D'(\R))\big).
\]
In other words, $\overline{u}$ is the unique piecewise distributional solution to (\ref{eq:wave1D}), (\ref{eq:ICwave1D}). Consequently, the whole
net $(u^\eps)_{\eps \in (0,1]}$ converges to $u = \overline{u}$.

If $u_0 \in {\mathcal C}^2(\R)$ and $u_1 \in {\mathcal C}^1(\R)$ do not have compact support, we consider an arbitrary rectangle $[-R,R]\times [0,T]$ and take a cut-off function $\chi(x)$ identically equal to 1 on a neighborhood of $[-R - T\max_{0\leq t \leq T} c(t), R + T\max_{0\leq t \leq T} c(t)]\times [0,T]$. By the derivation above, the Colombeau solution with initial data $\chi u_0$, $\chi u_1$ is associated with the piecewise distributional solution with the corresponding initial data. But both solutions coincide with the corresponding solutions without cut-off on the rectangle $[-R,R]\times [0,T]$, by finite propagation speed. This shows that the association result holds without the assumption of compact support.

In the next to last step, take $u_0, u_1 \in \D'(\R)$ and assume that they are distributions of finite order. Write $u_0 = u_0^- + u_0^+$, where $u_0^-$ has its support bounded from the right, and $u_0^+$ has support bounded from the left, and similarly for $u_1$. There is an integer $n$ such that $u_0^+\ast I_n \in {\mathcal C}^2(\R)$ and $u_1^+\ast I_n \in {\mathcal C}^1(\R)$, where $I_n$ is the $n$-fold convolution of the Heaviside function with itself; a similar assertion holds for $u_0^-\ast\check{I}_n$ and $u_0^-\ast\check{I}_n$. We may assume without loss that $u_0^- = u_1^- = 0$. Let $u$ and $U$ be the piecewise distributional solution and the Colombeau solution with initial data $u_0, u_1$, respectively.
Let $\widetilde{u}$ and $\widetilde{U}$ be the piecewise distributional solution and the Colombeau solution with initial data $u_0\ast I_n$ and $u_1\ast I_n$, respectively. By the previous step, $\widetilde{U}$ is associated with $\widetilde{u}$. But $\p_x^n\widetilde{U} = U$ and $\p_x^n\widetilde{u} = u$, since both satisfy the wave equation with initial data
$u_0, u_1$ in the appropriate settings. Therefore, $U$ is associated with $u$ as well.

Finally, if $u_0, u_1 \in \D'(\R)$ are arbitrary distributions, we use the fact that they are locally of finite order and argue by cut-off and finite propagation speed as above.
\end{proof}

\begin{remark}
The compactness argument in the first step of the proof can be set up in various spaces. For example, the uniform boundedness of $\int_\R|u^\eps_{xt}(t,x)|^2\,dx$ on $[0,T]$ implies
the equicontinuity of $(u^\eps)_{\eps \in (0,1]}$ in ${\mathcal C}\big([0,T):H^1(\R)\big)\subset {\mathcal C}\big([0,T):{\mathcal C}(\R)\big)$. Hence $(u^\eps)_{\eps \in (0,1]}$ is actually relatively compact in
${\mathcal C}\big([0,T):{\mathcal C}(\R)\big) = {\mathcal C}\big(\R\times[0,T)\big)$.
\end{remark}

\subsection*{Acknowledgments}
The first author expresses his most heartfelt thanks to G\"{u}nther H\"{o}rmann for the warm hospitality during his visit to the Fakult\"{a}t f\"{u}r Mathematik, Universit\"{a}t Wien from October 2, 2009 to March 31, 2011.
He also gives his deep appreciation to Michael Oberguggenberger for the warm hospitality during his several visits to Universit\"{a}t Innsbruck in the past years.

\end{document}